\documentclass[12pt]{amsart}
\usepackage{graphicx}
\usepackage{amsmath}
\usepackage{amssymb}
\usepackage{setspace}
\usepackage{amsthm}
\newtheorem{thm}{Theorem}[section]
\newtheorem{cor}[thm]{Corollary}

\newtheorem{lem}[thm]{Lemma}

\theoremstyle{definition}

\theoremstyle{remark}
\newtheorem{rem}[thm]{Remark}
\theoremstyle{conclusion}

\theoremstyle{conjecture}

\numberwithin{equation}{section}

\newcommand{\eps}{\varepsilon}

\begin{document}
\title[Uniform a priori estimates for higher order Lane-Emden equations]{Uniform a priori estimates for positive solutions of higher order Lane-Emden equations in $\mathbb{R}^n$}
\author{Wei Dai}
\author{Thomas Duyckaerts}

\address{Institut Universitaire de France, and LAGA, UMR 7539, Institut Galil\'{e}e, Universit\'{e} Paris 13, 93430 - Villetaneuse, France}
\email{duyckaer@math.univ-paris13.fr}

\address{School of Mathematics and Systems Science, Beihang University (BUAA), Beijing 100083, P. R. China, and LAGA, UMR 7539, Institut Galil\'{e}e, Universit\'{e} Paris 13, 93430 - Villetaneuse, France}
\email{weidai@buaa.edu.cn}

\thanks{W. Dai is supported by the Fundamental Research Funds for the Central Universities and the State Scholarship Fund of China (No. 201806025011). T. Duyckaerts is supported by the Institut Universitaire de France}

\begin{abstract}
In this paper, we study the existence of uniform a priori estimates for positive solutions to Navier problems of higher order Lane-Emden equations
\begin{equation*}
  (-\Delta)^{m}u(x)=u^{p}(x), \qquad \,\, x\in\Omega
\end{equation*}
for all large exponents $p$, where $\Omega\subset\mathbb{R}^{n}$ is a star-shaped or strictly convex bounded domain with $C^{2m-2}$ boundary, $n\geq4$ and $2\leq m\leq\frac{n}{2}$. Our results extend those of previous authors for second order $m=1$ to general higher order cases $m\geq2$.
\end{abstract}
\maketitle {\small {\bf Keywords:} Uniform a priori estimates, higher order Lane-Emden equations, Navier problems, positive solutions, Pohozaev identity, blow up. \\

}

\section{Introduction}

In this paper, we investigate the the following higher order Lane-Emden equations in bounded domain with Navier boundary conditions:
\begin{equation}\label{NHLE}\\\begin{cases}
(-\Delta)^{m}u(x)=u^{p}(x), \,\,\,\,\,\,\, u(x)>0, \,\,\,\,\,\,\,\, x\in\Omega, \\
u(x)=(-\Delta)u(x)=\cdots=(-\Delta)^{m-1}u(x)\equiv0,\,\,\,\,\,\,\,\, x\in\partial\Omega,
\end{cases}\end{equation}
where $1<p<+\infty$, $n\geq2$, $1\leq m\leq\frac{n}{2}$ and $\Omega\subset\mathbb{R}^{n}$ is a bounded domain with $C^{2m-2}$ boundary $\partial\Omega$. We assume the positive solutions $u\in C^{2m}(\Omega)\cap C^{2m-2}(\overline{\Omega})$.

The Lane-Emden equations of type \eqref{NHLE} have numerous important applications in conformal geometry and Sobolev inequalities. It also models many phenomena in mathematical physics and in astrophysics (see \cite{C,H}). We say equations \eqref{NHLE} have critical order if $m=\frac{n}{2}$ and non-critical order if $m<\frac{n}{2}$. The nonlinear terms in \eqref{NHLE} is called critical if $p=p_{c}:=\frac{n+2m}{n-2m}$ ($:=\infty$ if $m=\frac{n}{2}$) and subcritical if $1<p<p_{c}$.

When $m=1$, Ambrosetti and Rabinowitz \cite{AR} derived the existence of \emph{least energy} positive solution to \eqref{NHLE} for $1<p<p_{c}$ via variational minimization methods. When $m\geq2$, Chen, Fang and Li \cite{CFL}, Dai, Peng and Qin \cite{DPQ} (for $m<\frac{n}{2}$), Chen, Dai and Qin \cite{CDQ} (for $m=\frac{n}{2}$) derived a priori estimates, and hence existence of positive solutions to \eqref{NHLE} (via the Leray-Schauder fixed point theorem) for some restricted subranges of $p$ in $(1,p_{c})$. Subsequently, in \cite{DQ0} (for $m<\frac{n}{2}$) and \cite{DQ2} (for $m=\frac{n}{2}$), Dai and Qin established a priori estimates and existence of positive solutions for all $p\in\left(1,p_c\right)$, moreover, the positive solution $u$ to \eqref{NHLE} satisfies
\begin{equation}\label{lower-bound}
  \|u\|_{L^{\infty}(\overline{\Omega})}\geq\left(\frac{\sqrt{2n}}{diam\,\Omega}\right)^{\frac{2m}{p-1}}.
\end{equation}
The lower bounds \eqref{lower-bound} on the $L^{\infty}$-norm of positive solutions $u$ indicate that, if $diam\,\Omega<\sqrt{2n}$, then the $L^{\infty}$-norm must blow up as $p\rightarrow1+$.

\subsection{The non-critical order cases $1\leq m<\frac{n}{2}$.}

We first consider the non-critical order cases $1\leq m<\frac{n}{2}$. Using what is now known as Pohozaev identities, Pohozaev \cite{P} has shown that there are no positive solutions in the range $p_{c}<p<+\infty$ provided $\Omega$ is star-shaped. Han \cite{Han} and Rey \cite{R} proved that the $L^{\infty}$-norm of positive solutions of \eqref{NHLE} with $m=1$ blows up as $p\rightarrow p_{c}-$, in addition, they have also obtained the precise asymptotic behaviour for the least-energy solutions. Di \cite{D} established similar results as in \cite{Han,R} for the bi-harmonic case $m=2$ and strictly convex domain $\Omega$.

In this paper, we will prove that, if $\Omega$ is a star-shaped domain, the $L^{\infty}$-norm of positive solutions of \eqref{NHLE} with general $2\leq m<\frac{n}{2}$ blows up as $p\rightarrow p_{c}-$.

First, we will deduce a Pohozaev type variational identities (see \cite{EFJ,P,PS,PS1}) for the following generalized higher order Navier problems:
\begin{equation}\label{gNHLE}\\\begin{cases}
(-\Delta)^{m}u(x)=f(u(x)), \,\,\,\,\,\,\, u(x)\geq0, \,\,\,\,\,\,\,\, x\in\Omega, \\
u(x)=(-\Delta)u(x)=\cdots=(-\Delta)^{m-1}u(x)\equiv0,\,\,\,\,\,\,\,\, x\in\partial\Omega,
\end{cases}\end{equation}
where $n\geq2$, $m\geq1$, the function $f:\,\overline{\mathbb{R}_{+}}\rightarrow\overline{\mathbb{R}_{+}}$ is continuous and $\Omega\subset\mathbb{R}^{n}$ is a bounded domain with $C^{2m-2}$ boundary $\partial\Omega$.

The following variational identities are valid for higher order Navier problems \eqref{gNHLE}.
\begin{thm}\label{Pohozaev}
Suppose $u\in C^{2m}(\Omega)\cap C^{2m-2}(\overline{\Omega})$ is a nonnegative solution to \eqref{gNHLE}, then it satisfies the following identity:
\begin{eqnarray}\label{P}
   &&\int_{\Omega}\left[nF(u)-\frac{n-2m}{2}f(u)u\right]dx=\sum_{k=1}^{[\frac{m}{2}]}\int_{\partial\Omega}\left|\nabla(-\Delta)^{m-k}u\right|\,\left|\nabla(-\Delta)^{k-1}u\right|
  \left((x-x_{0})\cdot\nu\right)d\sigma_{x} \\
 \nonumber &&\qquad\qquad\qquad\qquad\qquad\qquad\qquad\quad
 +\left\{\frac{m}{2}\right\}\int_{\partial\Omega}\left|\nabla(-\Delta)^{\frac{m-1}{2}}u\right|^{2}\left((x-x_{0})\cdot\nu\right)d\sigma_{x},
\end{eqnarray}
where $x_{0}\in\Omega$ is arbitrary, $F(u):=\int_{0}^{u}f(t)dt$, $\nu$ denotes the unit outward normal vector at $x\in\partial\Omega$, $[a]$ denotes the largest integer lesser or equal to $a$ and $\{a\}:=a-[a]$.
\end{thm}

As a consequence of the Pohozaev type variational identities, we can deduce Liouville theorem for higher order Navier problems \eqref{gNHLE}. Liouville type results for fractional and higher order H\'{e}non-Hardy equations in balls with Dirichlet or Navier boundary condtions have been established in \cite{DQ0} by developing the method of scaling spheres. For Liouville theorems on higher order Dirichlet problems via Pohozaev type variational identities, please refer to \cite{EFJ,P,PS,PS1}.

Our Liouville type result for Navier problem \eqref{gNHLE} is the following corollary.
\begin{cor}\label{cor0}
Assume that the function $f$ satisfies $nF(t)-\frac{n-2m}{2}f(t)t\leq0$, and that $\Omega$ is star-shaped. Then Navier problem \eqref{gNHLE} has no nontrivial nonnegative solution.
\end{cor}

In particular, if we take $f(t):=t^{p}$, then Corollary \ref{cor0} implies immediately the following Liouville theorem for Navier problem \eqref{NHLE} in both critical and super-critical cases.
\begin{cor}\label{cor1}
Assume $1\leq m<\frac{n}{2}$, $p_{c}:=\frac{n+2m}{n-2m}\leq p<+\infty$ and $\Omega$ is star-shaped, then Navier problem \eqref{NHLE} has no nontrivial nonnegative solution.
\end{cor}

Corollary \ref{cor1} indicates there are no positive solution to Navier problem \eqref{NHLE} in star-shaped domain $\Omega$ when $p=p_{c}$. As a consequence of Corollary \ref{cor1} and the existence results for $p<p_{c}$ in \cite{DPQ,DQ0}, we can infer that the $L^{\infty}$-norm of solutions to \eqref{NHLE} in star-shaped domain $\Omega$ must blow up as $p\rightarrow p_{c}-$, or else one could derive a positive solution in the critical case $p=p_{c}$ via compactness arguments. We have the following corollary.
\begin{cor}\label{cor2}
Assume $\Omega$ is star-shaped, then any sequence of solutions $\{u_{p_{k}}\}$ to Navier problems \eqref{NHLE} with $p=p_{k}\rightarrow p_{c}$ must blow up in $L^{\infty}$-norm, that is,
\begin{equation*}\label{blow-up}
  \|u_{p_{k}}\|_{L^{\infty}(\overline{\Omega})}\rightarrow+\infty \quad\,\, \text{as} \,\,\, k\rightarrow\infty.
\end{equation*}
\end{cor}

\subsection{The critical order cases $m=\frac{n}{2}$ with $n\geq2$ even.}

Next, we consider the critical order cases $m=\frac{n}{2}$ with $n\geq2$ even.

In contrast with the non-critical order cases, Ren and Wei \cite{RW} showed that the least-energy solutions of \eqref{NHLE} stay uniformly bounded as $p\rightarrow+\infty$. Subsequently, Kamburov and Sirakov \cite{KS} proved that positive solutions of \eqref{NHLE} with $m=1$ in a 2D smooth bounded domain $\Omega$ are uniformly bounded for all large exponents $p_{0}\leq p<+\infty$. For asymptotic description of positive solutions to \eqref{NHLE} in the case $m=1$ and $n=2$ as $p\rightarrow+\infty$, please refer to \cite{AG,DIP,DIP1}.

In this paper, by using the methods from Kamburov and Sirakov \cite{KS} of employing the Green's representation formula and results from Chen, Dai and Qin \cite{CDQ}, we will establish uniform a priori estimates for positive solutions to critical order Navier problems \eqref{NHLE} (with general $m=\frac{n}{2}$ and $n\geq4$ even) for all large exponents $p$ in strictly convex domain $\Omega$ with $C^{n-2}$ boundary $\partial\Omega$.

We have the following uniform a priori estimates for the critical order Navier problems \eqref{NHLE}.
\begin{thm}\label{Thm0}
Assume $n\geq4$ is even, $m=\frac{n}{2}$, $\Omega\subset\mathbb{R}^{n}$ is strictly convex and let $p_{0}>1$. There exists a constant $C$ depending only on $p_{0}$, $n$ and $\Omega$, such that for all $p_{0}\leq p<+\infty$, any solution $u_{p}\in C^{n}(\Omega)\cap C^{n-2}(\overline{\Omega})$ to critical order problem \eqref{NHLE} satisfies:
\begin{equation*}
  \|u_{p}\|_{L^{\infty}(\overline{\Omega})}\leq C.
\end{equation*}
\end{thm}

\begin{rem}\label{rem0}
Theorem \ref{Thm0} extends the uniform a priori estimates derived in \cite{KS,RW} for second order case $m=1$ and $n=2$ to general critical order cases $m=\frac{n}{2}$ and $n\geq4$ is even.
\end{rem}

\begin{rem}\label{rem1}
Being essentially different from the second order case $m=1$ and $n=2$, the information and estimates on $-\Delta u$ play a crucial role in the proof of Theorem \ref{Thm0}, please see Lemma \ref{lem0} and \ref{lem2}. More precisely, we proved in Lemma \ref{lem2} the following crucial property:
\begin{equation*}
  \max_{\overline{\Omega}}(-\Delta)^{k}u\sim\frac{\left[\max_{\overline{\Omega}}u\right]^{\frac{2k}{n}p+(1-\frac{2k}{n})}}{p^{1-\frac{2k}{n}}}
\end{equation*}
for any $k=1,\cdots,\frac{n}{2}-1$. In particular, from the proof of Lemma \ref{lem2} (see \eqref{3-23}), one has the following pointwise estimates at the maximum $x_{0}$ of $u$ in $\overline{\Omega}$:
\begin{equation*}
  C''_{k}\frac{\left(u(x_{0})\right)^{\frac{2k}{n}p+(1-\frac{2k}{n})}}{p^{1-\frac{2k}{n}}}\leq(-\Delta)^{k}u(x_{0})\leq C'_{k}\frac{\left(u(x_{0})\right)^{\frac{2k}{n}p+(1-\frac{2k}{n})}}{p^{1-\frac{2k}{n}}}
\end{equation*}
for any $k=1,\cdots,\frac{n}{2}-1$. For related pointwise inequality in $\mathbb{R}^{n}$, we refer to Fazly, Wei and Xu \cite{FWX}.
\end{rem}

The rest of our paper is organized as follows. In Section 2, we carry out the proof for Theorem \ref{Pohozaev} and Corollary \ref{cor0}. Section 3 is devoted to proving our Theorem \ref{Thm0}.

\section{Proof of Theorem \ref{Pohozaev} and Corollary \ref{cor0}}

In this section, we will first prove Theorem \ref{Pohozaev}.
\begin{proof}[Proof of Theorem \ref{Pohozaev}]
Without loss of generality, we may assume that $0\in\Omega$ and $x_{0}=0$. Since $u$ is a nonnegative solution to the generalized Navier problem \eqref{gNHLE}, by Navier boundary condition and maximum principles, we get
\begin{equation}\label{2-0}
  (-\Delta)^{k}u(x)\geq0, \qquad \,\, \forall \, k=0,1,\cdots,m-1, \,\, x\in\Omega,
\end{equation}
moreover,
\begin{equation}\label{2-1}
  \nabla(-\Delta)^{k}u(x)=-\left|\nabla(-\Delta)^{k}u(x)\right|\nu, \qquad \,\, \forall \, k=0,1,\cdots,m-1, \,\, x\in\partial\Omega,
\end{equation}
where $\nu$ denotes the unit outward normal vector at $x\in\partial\Omega$.

Multiplying both sides of the equation \eqref{gNHLE} by $x\cdot\nabla u$ and integrating on $\Omega$, one gets
\begin{equation}\label{2-2}
  \int_{\Omega}(-\Delta)^{m}u(x)(x\cdot\nabla u)dx=\int_{\Omega}f(u(x))(x\cdot\nabla u)dx.
\end{equation}
Integrating by parts yields
\begin{equation}\label{2-3}
  RHS=\int_{\Omega}f(u(x))(x\cdot\nabla u)dx=\int_{\Omega}x\cdot\nabla F(u(x))dx=-n\int_{\Omega}F(u(x))dx,
\end{equation}
where $F(u):=\int_{0}^{u}f(t)dt$. For the left-hand side of \eqref{2-2}, by calculations and integrating by parts, we have
\begin{eqnarray}\label{2-4}
  && LHS=-\int_{\partial\Omega}\left(\nabla(-\Delta)^{m-1}u\cdot\nu\right)(x\cdot\nabla u)d\sigma_{x}+\int_{\Omega}\nabla(-\Delta)^{m-1}u\cdot\nabla udx \\
 \nonumber &&\qquad\quad\quad+\int_{\Omega}\nabla(-\Delta)^{m-1}u\cdot\nabla^{2}u\cdot xdx \\
 \nonumber &=& -\int_{\partial\Omega}\left|\nabla(-\Delta)^{m-1}u\right|\,|\nabla u|(x\cdot\nu)d\sigma_{x}+\int_{\Omega}f(u)udx-\int_{\Omega}(-\Delta)^{m-1}u\left[\nabla\Delta u\cdot x+\Delta u\right]dx \\
 \nonumber &=& -\int_{\partial\Omega}\left|\nabla(-\Delta)^{m-1}u\right|\,|\nabla u|(x\cdot\nu)d\sigma_{x}+2\int_{\Omega}f(u)udx+\int_{\Omega}(-\Delta)^{m-1}u\left(\nabla(-\Delta)u\cdot x\right)dx \\
 \nonumber &=& -\int_{\partial\Omega}\left|\nabla(-\Delta)^{m-1}u\right|\,|\nabla u|(x\cdot\nu)d\sigma_{x}-\int_{\partial\Omega}\left|\nabla(-\Delta)^{m-2}u\right|\,\left|\nabla(-\Delta)u\right|(x\cdot\nu)d\sigma_{x} \\
 \nonumber &&+4\int_{\Omega}f(u)udx+\int_{\Omega}(-\Delta)^{m-2}u\left(\nabla(-\Delta)^{2}u\cdot x\right)dx=\cdots\cdots.
\end{eqnarray}
Continuing this way, we get the following two different cases:\\
\emph{i)} if $m$ is even, then
\begin{align}\label{2-5}
  LHS=&-\sum_{k=1}^{\frac{m}{2}}\int_{\partial\Omega}\left|\nabla(-\Delta)^{m-k}u\right|\,\left|\nabla(-\Delta)^{k-1}u\right|(x\cdot\nu)d\sigma_{x}+m\int_{\Omega}f(u)udx \\
 \nonumber & +\int_{\Omega}(-\Delta)^{\frac{m}{2}}u\left(\nabla(-\Delta)^{\frac{m}{2}}u\cdot x\right)dx \\
 \nonumber =&-\sum_{k=1}^{\frac{m}{2}}\int_{\partial\Omega}\left|\nabla(-\Delta)^{m-k}u\right|\,\left|\nabla(-\Delta)^{k-1}u\right|(x\cdot\nu)d\sigma_{x}
  +\left(m-\frac{n}{2}\right)\int_{\Omega}f(u)udx;
\end{align}
\emph{ii)} if $m$ is odd, then
\begin{align}\label{2-6}
  LHS=&-\sum_{k=1}^{\frac{m-1}{2}}\int_{\partial\Omega}\left|\nabla(-\Delta)^{m-k}u\right|\,\left|\nabla(-\Delta)^{k-1}u\right|(x\cdot\nu)d\sigma_{x}+(m-1)\int_{\Omega}f(u)udx \\
 \nonumber & +\int_{\Omega}(-\Delta)^{\frac{m+1}{2}}u\left(\nabla(-\Delta)^{\frac{m-1}{2}}u\cdot x\right)dx \\
 \nonumber =&-\sum_{k=1}^{\frac{m-1}{2}}\int_{\partial\Omega}\left|\nabla(-\Delta)^{m-k}u\right|\,\left|\nabla(-\Delta)^{k-1}u\right|(x\cdot\nu)d\sigma_{x}+m\int_{\Omega}f(u)udx\\
 \nonumber &-\int_{\partial\Omega}\left|\nabla(-\Delta)^{\frac{m-1}{2}}u\right|^{2}(x\cdot\nu)d\sigma_{x}+\int_{\Omega}\nabla(-\Delta)^{\frac{m-1}{2}}u\cdot\nabla^{2}(-\Delta)^{\frac{m-1}{2}}u\cdot xdx\\
 \nonumber =&-\sum_{k=1}^{\frac{m-1}{2}}\int_{\partial\Omega}\left|\nabla(-\Delta)^{m-k}u\right|\,\left|\nabla(-\Delta)^{k-1}u\right|(x\cdot\nu)d\sigma_{x}
 +\left(m-\frac{n}{2}\right)\int_{\Omega}f(u)udx\\
 \nonumber &-\frac{1}{2}\int_{\partial\Omega}\left|\nabla(-\Delta)^{\frac{m-1}{2}}u\right|^{2}(x\cdot\nu)d\sigma_{x}.
\end{align}

As a consequence of \eqref{2-2}, \eqref{2-3}, \eqref{2-4}, \eqref{2-5} and \eqref{2-6}, we arrive at the Pohozaev type identity \eqref{P} immediately. This concludes our proof of Theorem \ref{Pohozaev}.
\end{proof}

\begin{proof}[Proof of Corollary \ref{cor0}]
Assume the function $f$ satisfies $nF(t)-\frac{n-2m}{2}f(t)t\leq0$. Without loss of generality, we may also assume $0\in\Omega$ and $\Omega$ is star-shaped with respect to $0$. We argue by contradiction, assuming that $u$ is a nontrivial nonnegative solution to Navier problem \eqref{gNHLE}. By \eqref{2-0} and strong maximum principle, we have
\begin{equation}\label{2-7}
  (-\Delta)^{k}u>0, \qquad \,\, \forall \, k=0,1,\cdots,m-1, \,\, x\in\Omega.
\end{equation}
Then, from \eqref{2-1}, Navier boundary condition and Hopf boundary Lemma, we deduce that
\begin{equation}\label{2-8}
  \left|\nabla(-\Delta)^{k}u(x)\right|=-\frac{\partial}{\partial\nu}(-\Delta)^{k}u(x)>0, \qquad \,\, \forall \, k=0,1,\cdots,m-1, \,\, x\in\partial\Omega.
\end{equation}

Since $\Omega$ is star-shaped w.r.t. $0$ and $f$ satisfies $nF(t)-\frac{n-2m}{2}f(t)t\leq 0$, then we can deduce from Theorem \ref{Pohozaev} and \eqref{2-8} that
\begin{multline}\label{2-9}
   0\geq \int_{\Omega}\left[nF(u)-\frac{n-2m}{2}f(u)u\right]dx\\
   =\sum_{k=1}^{[\frac{m}{2}]}\int_{\partial\Omega}\left|\nabla(-\Delta)^{m-k}u\right|\,\left|\nabla(-\Delta)^{k-1}u\right|
  \left(x\cdot\nu\right)d\sigma_{x}
 +\left\{\frac{m}{2}\right\}\int_{\partial\Omega}\left|\nabla(-\Delta)^{\frac{m-1}{2}}u\right|^{2}\left(x\cdot\nu\right)d\sigma_{x}>0,
\end{multline}
where we have used that $x\cdot\nu\geq 0$ for $x\in \partial\Omega$, and that this quantity cannot be identically zero on $\partial\Omega$.
It is clear that \eqref{2-9} is absurd. This finishes our proof of Corollary \ref{cor0}.
\end{proof}

\section{Proof of Theorem \ref{Thm0}}

In this section, we will prove Theorem \ref{Thm0} by using the methods from Kamburov and Sirakov \cite{KS} of employing the Green's representation formula and results from Chen, Dai and Qin \cite{CDQ}.

In the following, we will use $C$ to denote a general positive constant that may depend on $n$, $p_{0}$ and $\Omega$, and whose value may differ from line to line. In all the proof, we assume $p\geq p_0$.

For $n\geq4$ even and $m=\frac{n}{2}$, assume $u=u_{p}$ is a positive solution to critical order Navier problem \eqref{NHLE}, by maximum principle, we have $(-\Delta)^{k}u_{p}>0$ in $\Omega$ for any $k=0,1,\cdots,\frac{n}{2}-1$. Furthermore, using results from Chen, Dai and Qin \cite{CDQ}, we can prove the following Lemma.
\begin{lem}\label{lem0}
Assume $n\geq4$ is even, $m=\frac{n}{2}$, $\Omega$ is strictly convex and let $p_{0}>1$. There exist positive constants $\delta$ depending only on $\Omega$, and $C$ depending only on $n$, $p_{0}$ and $\Omega$, such that \\
(i) The maximum of the solution $u=u_{p}$ in $\overline{\Omega}$ can (only) be attained in $\overline{\Omega^{\delta}}:=\{x\in\Omega\,|\,dist(x,\partial\Omega)\geq\delta\}$, moreover, the maximum of $(-\Delta)^{k}u_{p}$ ($k=1,\cdots,\frac{n}{2}-1$) in $\overline{\Omega}$ can (only) be attained in $\overline{\Omega^{\delta}}$; \\
(ii) For every $p\geq p_{0}$, the solution $u=u_{p}$ satisfies the uniform bound:
\begin{equation*}
  \int_{\Omega}u^{p}(x)dx\leq C.
\end{equation*}
\end{lem}
\begin{proof}
\noindent\emph{Proof of (i).} By using the method of moving planes in local way as in \cite{CDQ} (see pp. 16-21 in \cite{CDQ}), we can get that (see pp. 19 in \cite{CDQ}), for any $x^{0}\in\partial\Omega$, there exists a $\delta_{0}>0$ depending only on $x^{0}$ and $\Omega$ such that, $u(x)$ is monotone increasing along the internal normal direction in the region
\begin{equation}\label{3-31}
  \overline{\Sigma_{\delta_{0}}}:=\left\{x\in\overline{\Omega}\,|\,0\leq(x-x^{0})\cdot\nu^{0}\leq\delta_{0}\right\}.
\end{equation}
Since $\partial\Omega$ is $C^{n-2}$, there exists a small enough $0<r_{0}<\frac{\delta_{0}}{8}$ depending only on $x^{0}$ and $\Omega$ such that, for any $x\in B_{r_{0}}(x^{0})\cap\partial\Omega$, $u(x)$ is monotone increasing along the internal normal direction at $x$ in the region
\begin{equation}\label{3-32}
  \overline{\Sigma_{x}}:=\left\{z\in\overline{\Omega}\,\Big|\,0\leq(z-x)\cdot\nu_{x}\leq\frac{3}{4}\delta_{0}\right\}.
\end{equation}
where $\nu_{x}$ denotes the unit internal normal vector at the point $x$ ($\nu_{x^{0}}:=\nu^{0}$). Since $x^{0}\in\partial\Omega$ is arbitrary and $\partial\Omega$ is compact, we can cover $\partial\Omega$ by finite balls $\{B_{r_{k}}(x^{k})\}_{k=0}^{K}$ with centers $\{x^{k}\}_{k=0}^{K}\subset\partial\Omega$ ($K$ depends only on $\Omega$). Let $\delta:=\frac{3}{4}\min\{\delta_{1},\cdots,\delta_{K}\}$ depending only on $\Omega$, then it is clear that for any $x\in\partial\Omega$,
\begin{equation}\label{3-34'}
  u(x+s\nu_{x}) \quad \text{is monotone increasing with respect to} \,\,\, s\in\left[0,\delta\right],
\end{equation}
and hence property (i) for $u=u_{p}$ follows from \eqref{3-34'} immediately.

Moreover, it is also clear from the procedure of moving planes in \cite{CDQ} that (see pp. 16-21 in \cite{CDQ}), $(-\Delta)^{k}u_{p}$ ($k=1,\cdots,\frac{n}{2}-1$) are also monotone increasing along the internal normal directions in the boundary layer $\overline{\Omega\setminus\overline{\Omega^{\delta}}}$, and hence the maximum of $(-\Delta)^{k}u_{p}$ ($k=1,\cdots,\frac{n}{2}-1$) in $\overline{\Omega}$ can (only) be attained in $\overline{\Omega^{\delta}}$.

\medskip

\noindent\emph{Proof of (ii).} Let $\lambda_{1}$ be the first eigenvalue for $(-\Delta)^{\frac{n}{2}}$ in $\Omega$ with Navier boundary condition, and $0<\phi\in C^n(\Omega)\cap C^{n-2}(\overline{\Omega})$ be the corresponding eigenfunction (without loss of generality, we may assume $\|\phi\|_{L^{\infty}(\overline{\Omega})}=1$), i.e.,
\begin{equation}\label{3-50}\\\begin{cases}
(-\Delta)^{\frac{n}{2}}\phi(x)=\lambda_{1}\phi(x) \,\,\,\,\,\,\,\,\,\, \text{in} \,\,\, \Omega, \\
\phi(x)=-\Delta \phi(x)=\cdots=(-\Delta)^{\frac{n}{2}-1}\phi(x)=0 \,\,\,\,\,\,\,\, \text{on} \,\,\, \partial\Omega.
\end{cases}\end{equation}
Then, since
\begin{equation*}
 \int_{\Omega}u^{p}\phi=\int (-\Delta)^{\frac n2} u\phi=\lambda_1\int_{\Omega} u\phi\leq \lambda_1 \left( \int_{\Omega} u^p\phi \right)^{\frac{1}{p}}\left(\int_{\Omega}\phi\right)^{\frac{1}{p'}},
\end{equation*}
we obtain, as in Lemma 3.2 in page 20 of \cite{CDQ},
\begin{equation}\label{3-52}
  \int_{\Omega}u^{p}(x)\phi(x)dx\leq\lambda^{p'}_{1}\int_{\Omega}\phi(x)dx\leq\lambda^{p'}_{1}|\Omega|.
\end{equation}
Thus, for any $p\geq p_{0}$, we have the following uniform bound:
\begin{equation}\label{3-53}
  \int_{\Omega}u^{p}(x)\phi(x)dx\leq C(n,p_{0},\Omega).
\end{equation}

Let $x\in \partial \Omega$. Since $\Omega$ is at least $C^1$, there exists a small $\eps_x>0$ and a neighborhood $V_x$ of $x$ in $\partial\Omega$ such that
$$W_x:=\left\{y+\delta \nu_x,\; y\in V_x,\; 0<\delta<\eps_x\right\}\subset \Omega.$$
Let
$$W_x':=\left\{y+\delta \nu_x,\; y\in V_x,\; \frac{\eps_x}{2}<\delta<\eps_x\right\}\subset \Omega,$$
Since $x\in\partial\Omega$ is arbitrary and $\partial\Omega$ is compact, we can find a finite subset $\{x^{k}\}_{k=0}^{K}$ of $\partial\Omega$ ($K$ depends only on $\Omega$), such that $\partial \Omega \subset \bigcup_{k=0}^K V_{x^k}.$
Considering the boundary layer $\overline{\Omega}_{\bar{\delta}}:=\{x\in\overline{\Omega}\,|\,dist(x,\partial\Omega)\leq\bar{\delta}\}$, we see that if $\bar{\delta}>0$ is small enough,
\begin{equation}
\label{3-T0}
\overline{\Omega}_{\bar{\delta}}\subset \bigcup_{k=0}^K W_{x^k},\quad \bigcup_{k=0}^K W_{x^k}'\subset \Omega\setminus \overline{\Omega}_{\bar{\delta}}.
\end{equation}
From the procedure of moving planes in \cite{CDQ}, for all $y\in V_{x^k}$, $\sigma\mapsto u(y+\sigma\nu_{x^k})$ is monotone increasing on $(0,\eps_{x^k})$, and thus, using the second inclusion in \eqref{3-T0}
\begin{equation}
 \label{3-T1}
 \int_{W_{x^k}} u^p(x)\,dx \leq 2\int_{W_{x^k}'}u^p(x)\,dx\leq 2\int_{\Omega\setminus \overline{\Omega}_{\bar{\delta}}} u^p(x)\,dx.
\end{equation}
As a consequence, using the first inclusion in \eqref{3-T0},
\begin{equation}
 \label{3-T2}
 \int_{\Omega} u^{p}(x)\,dx\leq \sum_{k=0}^K \int_{W_{x^k}}u^{p}(x)\,dx+\int_{\Omega\setminus \overline{\Omega}_{\bar{\delta}}}u^p(x)\,dx\leq (2K+3)\int_{\Omega\setminus \overline{\Omega}_{\bar{\delta}}} u^p(x)\,dx.
\end{equation}
Combining with the uniform bound \eqref{3-53}, we arrive at
\begin{equation}\label{3-54}
  \int_{\Omega}u^{p}(x)dx\leq
 (2K+3)\max_{x\in \Omega\setminus\overline{\Omega}_{\bar{\delta}}} {\frac{1}{\varphi(x)}}\int_{\Omega\setminus \overline{\Omega}_{\bar{\delta}}}\varphi u^p(x)\,dx\leq C(n,p_{0},\Omega),
\end{equation}
which proves property (ii). This completes our proof of Lemma \ref{lem0}.
\end{proof}

From now on, we will denote the solution $u_{p}$ by $u$ for the sake of simplicity.

Let $M:=\max_{\overline{\Omega}}u=\|u\|_{L^{\infty}(\overline{\Omega})}$, we aim to prove that, there exists a constant $C>0$ depending only on $n$, $p_{0}$ and $\Omega$, such that $M\leq C$ for any $p\geq p_{0}$. We may assume that $M>\max\left\{2^{n},2^{\frac{2n}{p_{0}-1}}\right\}$ hereafter, or else we have done.

We first rescale $u$, so that $\Omega\subseteq B_{\frac{1}{4}}(0)$. Indeed, let $R:=R(\Omega)>1$ be the smallest radius such that $\Omega\subseteq B_{\frac{R}{4}}(0)$. Then $u_{R}(x):=R^{\frac{n}{p-1}}u(Rx)$ is a nonnegative solution of Navier problem \eqref{NHLE} in $R^{-1}\Omega\subseteq B_{\frac{1}{4}}(0)$ and we only need to consider $u_{R}$ instead.
By Lemma \ref{lem0}, the maximum $M$ can (only) be attained at some point $x_{0}\in\overline{\Omega^{\delta}}:=\{x\in\Omega\,|\,dist(x,\partial\Omega)\geq\delta\}$. Without loss of generality, translating $\Omega$ if necessary, we may assume $0\in\overline{\Omega^{\delta}}$ and $x_{0}=0$, that is, $u(0)=M$. Note that after this translation, we have
$$\Omega \subseteq B_{\frac 12}(0).$$

For arbitrarily given $x\in\overline{\Omega^{\delta}}$, let $G(x,y)$ denote the Green's function for $(-\Delta)^{\frac{n}{2}}$ with pole at $x$. Then, we have
\begin{equation}\label{3-0}
  G(x,y)=C_{n}\ln\left(\frac{1}{|x-y|}\right)-h(x,y), \quad \forall \,\, y\in\overline{\Omega},
\end{equation}
where the $\frac{n}{2}$-harmonic function $h$ satisfies
\begin{equation}\label{3-8}\\\begin{cases}
(-\Delta)^{\frac{n}{2}}h(x,y)=0, \,\,\,\,\,\,\,\,\,\,\,\,\,\,\, y\in\Omega, \\
(-\Delta)^{k}h(x,y)=(-\Delta)^{k}\left(C_{n}\ln\frac{1}{|x-y|}\right),\,\,\,\,\, k=0,1,\cdots,\frac{n}{2}-1, \,\,\,\,\, y\in\partial\Omega.
\end{cases}\end{equation}
Since $\delta\leq|x-y|\leq1$ for all $y\in\partial\Omega$, we have
\begin{equation}\label{3-1}
  C\leq(-\Delta)^{\frac{n}{2}-1}h(x,y)=(-\Delta)^{\frac{n}{2}-1}\left(C_{n}\ln\frac{1}{|x-y|}\right)=\frac{C}{|x-y|^{n-2}}\leq\frac{C}{\delta^{n-2}}
\end{equation}
for any $y\in\partial\Omega$, and hence, the maximum principle implies
\begin{equation}\label{3-2}
  C\leq(-\Delta)^{\frac{n}{2}-1}h(x,y)\leq\frac{C}{\delta^{n-2}}, \quad \forall \,\, y\in\Omega.
\end{equation}
On the boundary $\partial\Omega$, we also have
\begin{equation}\label{3-3}
  C\leq(-\Delta)^{\frac{n}{2}-2}h(x,y)=(-\Delta)^{\frac{n}{2}-2}\left(C_{n}\ln\frac{1}{|x-y|}\right)=\frac{C}{|x-y|^{n-4}}\leq\frac{C}{\delta^{n-4}}
\end{equation}
for all $y\in\partial\Omega$. It follows from \eqref{3-2}, \eqref{3-3} and the maximum principles that
\begin{equation}\label{3-4}
  C\leq(-\Delta)^{\frac{n}{2}-2}h(x,y)\leq\frac{C}{\delta^{n-4}}+\frac{C}{\delta^{n-2}}\leq\frac{C}{\delta^{n-2}}, \quad \forall \,\, y\in\Omega.
\end{equation}
Continuing this way, we finally get
\begin{equation}\label{3-5}
  0\leq h(x,y)\leq C\ln\frac{1}{\delta}+\frac{C}{\delta^{n-2}}=:C, \quad \forall \,\, y\in\Omega.
\end{equation}
In conclusions, we have arrive at, for any given $x\in\overline{\Omega^{\delta}}$ and $k=0,\cdots,\frac{n}{2}-1$, there exist constants $C'_{k}$, $C''_{k}\geq0$ such that
\begin{equation}\label{3-5'}
  C'_{k}\leq(-\Delta)^{k}h(x,y)\leq C''_{k}, \quad \forall \,\, y\in\Omega.
\end{equation}

We have the following Lemma on uniform bound of the solution $u=u_{p}$.
\begin{lem}\label{lem1}
Assume $n\geq4$ is even, $m=\frac{n}{2}$, $\Omega$ is strictly convex and let $p_{0}>1$. For every $x\in\overline{\Omega^{\delta}}$ and $p\geq p_{0}$, the solution $u=u_{p}$ satisfies the uniform bound:
\begin{equation*}
  \frac{1}{M}\int_{\Omega}\ln\left(\frac{1}{|x-y|}\right)u^{p}(y)dy\leq C.
\end{equation*}
\end{lem}
\begin{proof}
By (ii) in Lemma \ref{lem0}, \eqref{3-5'} and the Green's representation formula, we have, for any $x\in\overline{\Omega^{\delta}}$ and $p\geq p_{0}$,
\begin{eqnarray}\label{3-6}
  M &\geq& u(x)=\int_{\Omega}G(x,y)u^{p}(y)dy \\
  \nonumber &=& C_{n}\int_{\Omega}\ln\left(\frac{1}{|x-y|}\right)u^{p}(y)dy-\int_{\Omega}h(x,y)u^{p}(y)dy \\
  \nonumber &\geq& C_{n}\int_{\Omega}\ln\left(\frac{1}{|x-y|}\right)u^{p}(y)dy-C\int_{\Omega}u^{p}(y)dy \\
  \nonumber &\geq& C_{n}\int_{\Omega}\ln\left(\frac{1}{|x-y|}\right)u^{p}(y)dy-C.
\end{eqnarray}
As a consequence, we get immediately that
\begin{equation}\label{3-6'}
  \frac{1}{M}\int_{\Omega}\ln\left(\frac{1}{|x-y|}\right)u^{p}(y)dy\leq\frac{M+C}{MC}\leq C.
\end{equation}
This finishes our proof of Lemma \ref{lem1}.
\end{proof}

Let $M_{k}:=\max_{\overline{\Omega}}(-\Delta)^{k}u=\|(-\Delta)^{k}u\|_{L^{\infty}(\overline{\Omega})}$ for $k=1,\cdots,\frac{n}{2}-1$. By Lemma \ref{lem0}, the maximum $M_{k}$ can (only) be attained at some point $x_{k}\in\overline{\Omega^{\delta}}:=\{x\in\Omega\,|\,dist(x,\partial\Omega)\geq\delta\}$, that is, $(-\Delta)^{k}u(x_{k})=M_{k}$.

We have the following Lemma which is crucial in our proof.
\begin{lem}\label{lem2}
Assume $n\geq4$ is even, $m=\frac{n}{2}$, $\Omega$ is strictly convex and let $p_{0}>1$. For every $k=1,\cdots,\frac{n}{2}-1$ and $p\geq p_{0}$, we have the following precise bound:
\begin{equation}\label{c-1}
  C''_{k}\frac{M^{\frac{2k}{n}p+(1-\frac{2k}{n})}}{p^{1-\frac{2k}{n}}}\leq M_{k}\leq C'_{k}\frac{M^{\frac{2k}{n}p+(1-\frac{2k}{n})}}{p^{1-\frac{2k}{n}}}.
\end{equation}
Moreover, we have, for any $p\geq p_{0}$,
\begin{equation}\label{c-2}
  0\leq M-u(x)\leq\frac{C}{p}M, \qquad \forall \,\, |x|\leq\frac{\delta}{\sqrt[n]{p}M^{\frac{p-1}{n}}}.
\end{equation}
\end{lem}
\begin{proof}
Since $M_{k}=(-\Delta)^{k}u(x_{k})$ and $x_{k}\in\overline{\Omega^{\delta}}$, by the Green's representation formula and \eqref{3-5'}, we have
\begin{eqnarray}\label{1-0}
  M_{k}&=&(-\Delta)^{k}u(x_{k})=C_{k}\int_{\Omega}\frac{1}{|x_{k}-y|^{2k}}u^{p}(y)dy-\int_{\Omega}(-\Delta)^{k}h(x_{k},y)u^{p}(y)dy \\
  \nonumber &\leq& C_{k}\int_{\Omega}\frac{1}{|x_{k}-y|^{2k}}u^{p}(y)dy.
\end{eqnarray}
Note that $B_{\delta}(x_{k})\subseteq\Omega$. For every $p\geq p_{0}$,
\begin{equation}\label{1-1}
  \int_{|x_{k}-y|\leq\frac{\delta}{p^{\frac{1}{n}}M^{\frac{p-1}{n}}}}\frac{1}{|x_{k}-y|^{2k}}u^{p}(y)dy\leq C_{k}\frac{M^{p}}{p^{1-\frac{2k}{n}}M^{(1-\frac{2k}{n})(p-1)}}=C_{k}\frac{M^{\frac{2k}{n}p+(1-\frac{2k}{n})}}{p^{1-\frac{2k}{n}}},
\end{equation}
and, by (ii) in Lemma \ref{lem0},
\begin{eqnarray}\label{1-2}
  \int_{\Omega\cap\Big\{|x_{k}-y|\geq\frac{p^{\frac{1}{2k}-\frac{1}{n}}\delta}{M^{\frac{p}{n}+(\frac{1}{2k}-\frac{1}{n})}}\Big\}}\frac{1}{|x_{k}-y|^{2k}}u^{p}(y)dy
  &\leq&\left(\frac{M^{\frac{p}{n}+(\frac{1}{2k}-\frac{1}{n})}}{p^{\frac{1}{2k}-\frac{1}{n}}\delta}\right)^{2k}\int_{\Omega}u^{p}(y)dy \\
  \nonumber &\leq& C_{k}\frac{M^{\frac{2k}{n}p+(1-\frac{2k}{n})}}{p^{1-\frac{2k}{n}}}.
\end{eqnarray}
In the case $\frac{1}{p^{\frac{1}{n}}M^{\frac{p-1}{n}}}<\frac{p^{\frac{1}{2k}-\frac{1}{n}}}{M^{\frac{p}{n}+(\frac{1}{2k}-\frac{1}{n})}}$, we can also deduce from Lemma \ref{lem1} that, for every $p\geq p_{0}$,
\begin{eqnarray}\label{1-3}
  && \int_{\frac{\delta}{p^{\frac{1}{n}}M^{\frac{p-1}{n}}}\leq|x_{k}-y|\leq\frac{p^{\frac{1}{2k}-\frac{1}{n}}\delta}{M^{\frac{p}{n}+(\frac{1}{2k}-\frac{1}{n})}}}
  \frac{1}{|x_{k}-y|^{2k}}u^{p}(y)dy\\
 \nonumber &\leq&\left[\frac{1}{M}\int_{\Omega}\ln\left(\frac{1}{|x_{k}-y|}\right)u^{p}(y)dy\right]\frac{M}{\left(\frac{\delta}{p^{\frac{1}{n}}M^{\frac{p-1}{n}}}\right)^{2k}
  \ln\left(\frac{M^{\frac{p}{n}+(\frac{1}{2k}-\frac{1}{n})}}{p^{\frac{1}{2k}-\frac{1}{n}}\delta}\right)} \\
 \nonumber &\leq& C_{k}\frac{M^{1+\frac{2k}{n}(p-1)}p^{\frac{2k}{n}}}{\left(\frac{p}{n}+\frac{1}{2k}-\frac{1}{n}\right)\ln M}
 \leq C_{k}\frac{M^{\frac{2k}{n}p+(1-\frac{2k}{n})}}{p^{1-\frac{2k}{n}}},
\end{eqnarray}
where at the last line we have used $M>\max\left\{2^{n},2^{\frac{2n}{p_{0}-1}}\right\}$. Combining \eqref{1-0}, \eqref{1-1}, \eqref{1-2} and \eqref{1-3}, we get
\begin{equation}\label{1-4}
  M_{k}=(-\Delta)^{k}u(x_{k})\leq C'_{k}\frac{M^{\frac{2k}{n}p+(1-\frac{2k}{n})}}{p^{1-\frac{2k}{n}}}.
\end{equation}

Since $B_{\delta}(0)\subseteq\Omega$ and $u(0)=M$, by \eqref{1-4} with $k=1$ and applying the inhomogeneous Harnack inequality (see Theorems 9.20 and 9.22 in \cite{GT} or Theorem 4.17 in \cite{HL}), we get
\begin{equation}\label{3-21}
  0\leq u(0)-u(x)\leq CM_{1}r^{2}\leq C\frac{M^{\frac{2}{n}p+(1-\frac{2}{n})}}{p^{1-\frac{2}{n}}}r^{2}, \quad \forall \,\, x\in B_{r}(0)
\end{equation}
for all $r\in \left[0,\frac{\delta}{4}\right]$. Indeed, since $B_{4r}(0)\subseteq\Omega$, by Theorem 9.22 in \cite{GT}, we have, there exists a $q$ depending only on $n$ such that
\begin{equation*}
    \left(\frac{1}{|B_{2r}(0)|} \int_{B_{2r}(0)}\left(u(0)-u(x)\right)^{q}dx\right)^{\frac{1}{q}}\leq C\left(\inf_{x\in B_{2r}(0)} \big(u(0)-u(x)\big)+r\|\Delta u\|_{L^{n}(B_{2r}(0))}\right).
\end{equation*}
Combining this with Theorem 9.20 in \cite{GT}, we deduce that
\begin{eqnarray*}
  \sup_{x\in B_r(0)}\left(u(0)-u(x)\right)&\leq& C\left(\left(\frac{1}{|B_{2r}(0)|} \int_{B_{2r}(0)} (u(0)-u(x))^{q}dx\right)^{\frac{1}{q}}+r\|\Delta u\|_{L^{n}(B_{2r}(0))}\right) \\
  &\leq& C\left(\inf_{x\in B_{2r}(0)}\big(u(0)-u(x)\big)+r\|\Delta u\|_{L^{n}(B_{2r}(0))}\right),
\end{eqnarray*}
which yields \eqref{3-21} immediately.

The inequality \eqref{3-21} implies, for any $p\geq p_{0}$,
\begin{equation}\label{3-22}
  0\leq M-u(x)\leq\frac{C}{p}M, \qquad \forall \,\, |x|\leq\frac{\delta}{\sqrt[n]{p}M^{\frac{p-1}{n}}}.
\end{equation}

Combining (ii) in Lemma \ref{lem0}, the Green's representation formula, \eqref{3-5'} and \eqref{3-22} yield that, for any $p\geq p_{0}$,
\begin{eqnarray}\label{3-23}
   M_{k}&\geq& (-\Delta)^{k}u(0)=C_{k}\int_{\Omega}\frac{1}{|y|^{2k}}u^{p}(y)dy-\int_{\Omega}(-\Delta)^{k}h(0,y)u^{p}(y)dy \\
  \nonumber &\geq& C_{k}\int_{|x|\leq\frac{\delta}{p^{\frac{1}{n}}M^{\frac{p-1}{n}}}}
   \frac{1}{|x|^{2k}}\left(1-\frac{C}{p}\right)^{p}M^{p}dx-\widetilde{C}_{k} \\
  \nonumber &\geq& C_{k}M^{p}\int_{0}^{\frac{\delta}{p^{\frac{1}{n}}M^{\frac{p-1}{n}}}}r^{n-1-2k}dr-\widetilde{C}_{k}\\
 \nonumber &\geq&C''_{k}\frac{M^{\frac{2k}{n}p+(1-\frac{2k}{n})}}{p^{1-\frac{2k}{n}}}.
\end{eqnarray}
This concludes our proof of Lemma \ref{lem2}.
\end{proof}

Since $0\in\overline{\Omega^{\delta}}$, combining Lemma \ref{lem1} and Lemma \ref{lem2} yields that, for any $p\geq p_{0}$,
\begin{eqnarray}\label{3-14}
   C&\geq& \frac{1}{M}\int_{|x|\leq\frac{\delta}{M^{\frac{p-1}{n}}p^{\frac{1}{n}}}}
   \ln\left(\frac{1}{|x|}\right)\left(1-\frac{C}{p}\right)^{p}M^{p}dx \\
  \nonumber &\geq& CM^{p-1}\int_{0}^{\frac{\delta}{M^{\frac{p-1}{n}}p^{\frac{1}{n}}}}\ln\left(\frac{1}{r}\right)r^{n-1}dr\\
 \nonumber &\geq&\frac{CM^{p-1}}{M^{p-1}p}\ln\left(\frac{M^{\frac{p-1}{n}}p^{\frac{1}{n}}}{\delta}\right)\\
 \nonumber &\geq& C\ln M,
\end{eqnarray}
which implies immediately the desired uniform a priori estimate:
\begin{equation}\label{3-12}
  M\leq e^{C}.
\end{equation}
This concludes our proof of Theorem \ref{Thm0}.


\begin{thebibliography}{99}

\bibitem{AG} Adimurthi and M. Grossi, {\it Asymptotic estimates for a two-dimensional problem with polynomial nonlinearity}, Proc. Amer. Math. Soc., \textbf{132} (2004), 1013-1019.

\bibitem{AR} A. Ambrosetti and P. Rabinowitz, {\it Dual variational methods in critical point theory and applications}, J. Funct. Anal., \textbf{14}(1973), no. 4, 349-381.

\bibitem{C} S. Chandrasekhar, {\it An introduction to the study of stellar structure}, vol. 2, Courier Corporation, 1957.

\bibitem{CDQ} W. Chen, W. Dai and G. Qin, {\it Liouville type theorems, a priori estimates and existence of solutions for critical order Hardy-H\'{e}non equations in $R^n$}, preprint, submitted for publication, arXiv: 1808.06609.

\bibitem{CFL} W. Chen, Y. Fang and C. Li, {\it Super poly-harmonic property of solutions for Navier boundary problems on a half space}, J. Funct. Anal., \textbf{265} (2013), 1522-1555.

\bibitem{D} Geng Di, {\it On blow-up of positive solutions for a biharmonic equation involving nearly critical exponent}, Commun. in PDEs, \textbf{24} (1999), no. 11-12, 1451-1467.

\bibitem{DIP} F. De Marchis, I. Ianni and F. Pacella, {\it Asymptotic analysis and sign-changing bubble towers for Lane-Emden problems}, J. Eur. Math. Soc., \textbf{17} (2015), no. 8, 2037-2068.

\bibitem{DIP1} F. De Marchis, I. Ianni and F. Pacella, {\it Asymptotic profile of positive solutions of Lane-Emden problems in dimension two}, Journal of Fixed Point Theory and Applications, \textbf{19} (2017), no. 1, 889-916.

\bibitem{DPQ} W. Dai, S. Peng and G. Qin, {\it Liouville type theorems, a priori estimates and existence of solutions for non-critical higher order Lane-Emden-Hardy equations}, preprint, submitted for publication, arXiv: 1808.10771.

\bibitem{DQ0} W. Dai and G. Qin, {\it Liouville type theorems for fractional and higher order H\'{e}non-Hardy type equations via the method of scaling spheres}, preprint, submitted for publication, arXiv: 1810.02752.

\bibitem{DQ2} W. Dai and G. Qin, {\it Liouville type theorem for critical order H\'{e}non-Lane-Emden equations on a half space and its applications}, preprint, submitted for publication, arXiv: 1811.00881.

\bibitem{EFJ} D. E. Edmunds, D. Fortunato and E. Jannelli, {\it Critical exponents, critical dimensions and the biharmonic operator}, Arch. Rational Mech. Anal., \textbf{112} (1990), no. 3, 269-289.
    
\bibitem{FWX} M. Fazly, J. Wei and X. Xu, {A point-wise inequality for the fourth order Lane-Emden equation}, Analysis \& PDE, \textbf{8} (2015), no. 7, 1541-1563.

\bibitem{GT} D. Gilbarg and N. S. Trudinger, {\it Elliptic partial differential equations of second order}, 2nd Ed., Springer-Verlag, New York, 1983.

\bibitem{H} G. P. Horedt, {\it Exact solutions of the Lane-Emden equation in $N$-dimensional space}, Astronomy and Astrophysics, \textbf{160} (1986), 148-156.

\bibitem{Han} Z. Han, {\it Asymptotic approach to singular solutions for nonlinear elliptic equations involving critical Sobolev exponent}, Annales de l'Institut Henri Poincar\'{e} C: Analyse non lin\'{e}aire, \textbf{8} (1991), 159-174.

\bibitem{HL} Q. Han and F. Lin, {\it Elliptic partial differential equations}, vol. 1 of Courant Lect. Notes in Math., Courant Inst. of Math. Sci., New York; Amer. Math. Soc., Providence, RI, 2nd Ed., 2011.

\bibitem{KS} N. Kamburov and B. Sirakov, {\it Uniform a priori estimates for positive solutions of the Lane-Emden equation in the plane}, Calc. Var. \& PDEs, \textbf{57} (2018): 164.

\bibitem{P} S. I. Pohozaev, {\it Eigenfunctions of the equation $\Delta u+\lambda f(u)=0$}, Dokl. Akad. Nauk SSSR, \textbf{165}(1965), no. 1, 36-39.

\bibitem{PS} P. Pucci and J. Serrin, {\it A general variational identity}, Indiana Univ. Math. J., \textbf{35} (1986), 681-703.

\bibitem{PS1} P. Pucci and J. Serrin, {\it Critical exponents and critical dimensions for poly-harmonic operators}, J. Math. Pures et Appl., \textbf{69} (1990), 55-83.

\bibitem{R} O. Rey, {\it The role of the Green's function in a non-linear elliptic equation involving the critical Sobolev exponent}, J. Funct. Anal., \textbf{89} (1990), no.1, 1-52.

\bibitem{RW} X. Ren and J. Wei, {\it On a two-dimensional elliptic problem with large exponent in nonlinearity}, Trans. Amer. Math. Soc., \textbf{343} (1994), no. 2, 749-763.

\end{thebibliography}
\end{document}